\documentclass[times]{amsart}

\usepackage{amssymb,amsmath}
\usepackage[colorlinks,plainpages,backref]{hyperref}
\usepackage{mathrsfs}
\usepackage[neveradjust]{paralist}
\usepackage{epic,eepic}
\usepackage{url,epsfig,psfrag,graphicx}
\usepackage{float,wrapfig}

\restylefloat{figure}

\theoremstyle{definition}
\newtheorem{ntn}{Notation}[section]
\newtheorem*{notn}{Notation}
\newtheorem{dfn}[ntn]{Definition}

\theoremstyle{plain}
\newtheorem{lem}[ntn]{Lemma}

\newtheorem{thm}[ntn]{Theorem}
\newtheorem{cor}[ntn]{Corollary}

\newtheorem{cnv}[ntn]{Convention}
\theoremstyle{remark}
\newtheorem*{clm}{Claim}

\newtheorem{rmk}[ntn]{Remark}
\newtheorem{exa}[ntn]{Example}

\newcommand{\boldone}{{\mathbf{1}}}
\newcommand{\bolda}{{\mathbf{a}}}
\newcommand{\boldb}{{\mathbf{b}}}
\newcommand{\boldc}{{\mathbf{c}}}

\newcommand{\boldp}{{\mathbf{p}}}
\newcommand{\boldq}{{\mathbf{q}}}
\newcommand{\boldr}{{\mathbf{r}}}
\newcommand{\boldu}{{\mathbf{u}}}
\newcommand{\boldv}{{\mathbf{v}}}

\newcommand{\del}{\partial}

\newcommand{\de}{{\mathrm d}}

\newcommand{\eps}{\varepsilon}

\newcommand{\onto}{\twoheadrightarrow}
\renewcommand{\to}{\longrightarrow}

\newcommand{\calE}{{\mathscr{E}}}
\newcommand{\calF}{\mathscr{F}}

\newcommand{\calH}{\mathscr{H}}

\newcommand{\calK}{\mathscr{K}}

\newcommand{\CC}{\mathbb{C}}

\newcommand{\NN}{\mathbb{N}}

\newcommand{\QQ}{\mathbb{Q}}
\newcommand{\RR}{\mathbb{R}}
\newcommand{\ZZ}{\mathbb{Z}}

\renewcommand{\bar}{\overline}

\DeclareMathOperator{\gr}{gr}

\DeclareMathOperator{\Hom}{Hom}

\DeclareMathOperator{\Li}{Li}

\DeclareMathOperator{\Proj}{Proj}
\DeclareMathOperator{\qdeg}{qdeg}
\DeclareMathOperator{\rk}{rk}

\DeclareMathOperator{\Spec}{Spec}

\DeclareMathOperator{\Tor}{Tor}
\DeclareMathOperator{\Var}{Var}

\def\schluss{\hfill\ensuremath{\diamond}}

\def\comment#1{}

\begin{document}

\title
{On the $b$-functions of hypergeometric systems
}

\author{Thomas Reichelt, Christian Sevenheck, Uli Walther}


\address{T.~Reichelt\\
Universit\"at Heidelberg\\
Mathematisches Institut\\
Im Neuenheimer Feld 205\\
69120 Heidelberg\\
Germany
}
\email{reichelt@mathi.uni-heidelberg.de}
\address{C.\ Sevenheck\\
Technische Universit\"at Chemnitz\\
Fakult\"at f\"ur Mathematik\\
09107 Chemnitz\\
Germany
}
\email{christian.sevenheck@mathematik.tu-chemnitz.de }
\address{
U.\ Walther\\
Purdue University\\
Dept.\ of Mathematics\\
150 N.\ University St.\\
West Lafayette, IN 47907\\
USA}
\email{walther@math.purdue.edu}


%



\begin{abstract}
For any integer $d\times (n+1)$ matrix $A$ and parameter $\beta\in\CC^d$
let $M_A(\beta)$ be the associated $A$-hypergeometric (or GKZ) system
in the variables $x_0,\ldots,x_n$. We describe bounds for the (roots of the)
$b$-functions of both $M_A(\beta)$ and its Fourier
transform along the hyperplanes $(x_j=0)$. We also give an estimate
for the $b$-function for restricting $M_A(\beta)$ to a  generic point.
\end{abstract}


\thanks{TR was supported by a DFG Emmy-Noether-Fellowship (RE 3567/1-1),
TR and CS acknowledge partial support by the ANR/DFG joint program
SISYPH (ANR-13-IS01-0001-01 and SE~1115/5-1). UW was supported by the
NSF under grant 1401392-DMS.}

\maketitle

\setcounter{tocdepth}{1}

\tableofcontents

\numberwithin{equation}{section}

\bigskip

Let $D$ be the ring of algebraic $\CC$-linear differential
operators on $\CC^{n+1}$ with coordinates $x_0,\ldots,x_n$.

\begin{dfn}[Compare {\cite{Kashiwara-bfu,MaisonobeMebkhout-SMF04}}]
\label{dfn-bfu-hyperplane}
Let $M$ be a left $D$-module and pick an element $m\in M$ with
annihilator $I\subseteq D$.  If $(V^iD)$ is the vector space spanned
by the monomials $x^\alpha\del^\beta$ with $\alpha_0-\beta_0\geq i$
then the \emph{$b$-function} of $m\in M$ along the coordinate
hyperplane $x_0=0$ is the minimal monic polynomial $b(s)$ that satisfies:
$b(x_0\del_0)\cdot m\in (V^1D)\cdot m$ in $M$, which is to say
$b(x_0\del_0)\in I+(V^1D)$ in $D$. 

If $M$ is cyclic, i.e., $M=D/I$, then we call \emph{$b$-function 
of $M$} the $b$-function in the above sense of the element $1+I\in M$.
\end{dfn}

The $b$-function exists in greater generality along any hypersurface
$(f=0)$, as long as the module $M$ is holonomic,
cf.~\cite{Kashiwara-bfu}. The $V$-filtration of Kashiwara and
Malgrange then takes the form $(V^iD)=\{P\in D\mid f^{i+k} \text{
  divides } P\bullet f^k\text{ for }k\gg 0\}$.  Both the
$V$-filtration and the $b$-function are intimately connected to the
restriction of the given $D$-module to the hypersurface.  The purpose
of this note is to give, for any $A$-hypergeometric system as well as
its Fourier transform, an explicit arithmetic description of a bound
for the root set of the $b$-function along any coordinate hyperplane
that involves the parameter $\beta$ in a very elementary way.

We have several applications in mind: first, it is a longstanding
question to understand the monodromy of $A$-hypergeometric systems,
and for this purpose the roots of the $b$-function as considered above
can be of some use.  On the other hand, the Fourier transform of an
$A$-hypergeometric system often (see \cite{SchulzeWalther-ekdi})
appears as a direct image module under a natural torus embedding given
by the columns of the matrix $A$.  This point of view turns out to be
extremely useful for Hodge theoretic considerations of
$A$-hypergeometric systems (see \cite{Reichelt-Comp14}). It is one of
the fundamental insights of Morihiko Saito (see \cite[Section
  3.2]{Saito1}) that the boundary behavior of variations of Hodge
structures (or, more generally, of mixed Hodge modules) is controlled
by the Kashiwara--Malgrange filtration along such a boundary divisor.
In the case of a cyclic $D$-module, such as $A$-hypergeometric systems
or their Fourier transforms, one can often deduce a large part of this
filtration from the values of the $b$-function.  
We refer to
\cite{ReicheltSevenheck-15} for an immediate application of our
results.  In a third direction, one can also see our calculation of the
$b$-function of the Fourier transform as a refinement of
\cite{SchulzeWalther-ekdi,FernandezWalther-PAMS11} geared towards
restriction of $A$-hypergeometric systems.

In the last part we compute an upper bound for the $b$-function of
restriction of the $A$-hypergeometric system to a generic point, again
in elementary terms of $A$ and $\beta$. Since the restriction of a
$D$-module to a point is a dual object to the $0$-th level
solution functor, our estimate can be viewed as a step towards a
sheafification in $\beta$ of the solution space, a problem that
remains unsolved.

\subsubsection*{Acknowledgements}
We would like to thank
the Forschungsinstitut Oberwolfach for hosting us in April of 2015.
\smallskip

\noindent We are greatly indebted to an unknown referee for very careful
reading, pointing out a number of misprints.

\section{Basic notions and results}

\begin{notn}
Throughout, the base field is $\CC$ and we consider a $\CC$-vector
space $V$ of dimension $n+1$.
\end{notn}

In this introductory section we review basic facts on
$A$-hypergeometric systems as well as the Euler--Koszul
functor. Readers are advised to refer to \cite{MMW} for more detailed
explanations.

\begin{ntn}
For any integer matrix $A$, let $R_A$ (resp.\ $O_A$) be
the polynomial ring over $\CC$ generated by the variables $\del_j$
(resp.\ $x_j$)
corresponding to the columns $\bolda_j$ of $A$. We identify $O_A$ with
the symmetric algebra on $\Hom_\CC(V,\CC)\cong \bigoplus \CC\cdot x_j$.
Further, let $D_A$ be the ring of $\CC$-linear differential operators
on $O_A$, where we identify $\frac{\del}{\del x_j}$ with $\del_j$
and multiplication by $x_j$ with $x_j$
so that both $R_A$ and $O_A$ become subrings of $D_A$.
\end{ntn}

\subsection{$A$-hypergeometric systems}

Let $A=(\bolda_0,\ldots,\bolda_n)$ be an integer $d\times (n+1)$
matrix, $d\le n+1$.  For convenience we assume that $\ZZ A=\ZZ^d$. For
$(v_1,\ldots,v_r)=\boldv\in \ZZ^r$ we denote by
$\boldv_+,\boldv_-$ the vectors given by
\[
(\boldv_+)_j=\max(0,v_j) \qquad\text{and}\qquad
(\boldv_-)_j=\max(0,-v_j).
\]

For the complex parameter vector $\beta\in\CC^d$ consider the
system of $d$ \emph{homogeneity equations}
\begin{eqnarray}\label{eqn-Euler}
E_i\bullet\phi
  &=&\beta_i\cdot\phi,
\end{eqnarray}
where $E_i=\sum_{j=0}^n a_{i,j}x_j\del_j$ is the $i$-th \emph{Euler operator},
together with the \emph{toric} (partial differential) \emph{equations}
\begin{eqnarray}\label{eqn-toric}
\{(\underbrace{\del^{\boldv_+}-\del^{\boldv_-}}_{:=\Delta_\boldv})\bullet
\phi=0&\mid& A\cdot \boldv=0\}.
\end{eqnarray}
In $R_A$, the toric operators $\{\Delta_\boldv|A\cdot\boldv=0\}$
generate the \emph{toric ideal} $I_A$. The quotient
\[
S_A:=R_A/I_A
\]
is
naturally isomorphic to the semigroup ring $\CC[\NN A]$. In $D_A$, the left
ideal generated by all equations \eqref{eqn-Euler} and
\eqref{eqn-toric} is the \emph{hypergeometric ideal} $H_A(\beta)$.  We
put
\[
M_A(\beta):=D_A/H_A(\beta);
\]
this is the $A$-hypergeometric system introduced and first
investigated by Gelfand, Graev, Kapranov, and Zelevinsky,
in \cite{Gelfand-Nauk86} and a string of other papers.  \schluss

\subsection{$A$-degrees}

If the rowspan of $A$ contains $\boldone_A$ we call $A$
\emph{homogeneous}. Homogeneity is equivalent to $I_A$ defining a
projective variety, and also to the system $H_A(\beta)$ having only regular
singularities \cite{Hotta,SchulzeWalther-Duke}.  A more general
\emph{$A$-degree function} on $R_A$ and $D_A$ is induced by:
\[
-\deg_A(x_j):=\bolda_j=:\deg_A(\del_j).
\]
We denote $\deg_{A,i}(-)$ the
$A$-degree function associated to the weight given by the $i$-th row of
$A$, so $\deg_A=(\deg_{A,1},\ldots,\deg_{A,d})$.

An $R_A$- (resp.\ $D_A$-)module $M$ is \emph{$A$-graded} if it has a
decomposition $M=\bigoplus_{\alpha\in\ZZ^d} M_\alpha$ such that the
module structure respects the grading $\deg_A(-)$ on $R_A$
(resp.\ $D_A$) and $M$.  If $N$ is an $A$-graded $R_A$-module, then we
denote $\deg_A(N)\subseteq \ZZ^d$ the set of all degrees of all
non-zero homogeneous elements of $N$. The \emph{quasi-degrees}
$\qdeg_A(N)$ of $N$ are the points in the Zariski closure in $\CC^d$
of $\deg_A(N)$.

As is common, if $M$ is
$A$-graded then $M(\boldb)$ denotes for each $\boldb\in\ZZ A$ its
shift with graded structure
$(M(\boldb))_{\boldb'}=M_{\boldb+\boldb'}$.

\subsection{Euler--Koszul complex}

Since
\begin{eqnarray*}
    x^\boldu E_i-E_ix^\boldu  &=& -(A\cdot \boldu)_ix^\boldu,\\
    \del^\boldu E_i-E_i\del^\boldu &=& (A\cdot \boldu)_i\del^\boldu,
\end{eqnarray*}
we have
\begin{eqnarray}\label{eq-deg}
E_iP=P(E_i-\deg_{A,i}(P))
\end{eqnarray}
for any $A$-homogeneous $P\in D_A$ and all $i$.

On the $A$-graded $D_A$-module $M$ one can thus define
commuting $D_A$-linear endomorphisms $E_i$ via
\[
E_i\circ m:=(E_i+\deg_{A,i}(m))\cdot m
\]
for $A$-homogeneous elements $m\in M$.  In particular, if $N$ is an
$A$-graded $R_A$-module one obtains commuting sets of
$D_A$-endomorphisms on the left $D_A$-module $D_A\otimes_{R_A}N$ by
\[
E_i\circ (P\otimes Q):=(E_i+\deg_{A,i}(P)+\deg_{A,i}(Q))P\otimes Q.
\]

The \emph{Euler--Koszul complex} $\calK_\bullet(N;\beta)$ of the
$A$-graded $R_A$-module $N$ is the homological Koszul complex induced
by $E-\beta:=\{(E_i-\beta_i)\circ\}_1^d$ on $D_A\otimes_{R_A}N$. 
In
particular, the terminal module $D_A\otimes_{R_A}N$ sits in
homological degree zero.  We denote the homology groups of
$\calK_\bullet(N;\beta)$ by $\calH_\bullet(N;\beta)$.
Implicit in the notation is ``$A$'': different presentations of
semigroup rings that act on $N$ yield different Euler--Koszul
complexes. 

If $N(\boldb)$ denotes the usual
shift-of-degree functor on the category of graded $R_A$-modules, then
$\calK_\bullet(N;\beta)(\boldb)$ and
$\calK_\bullet(N(\boldb);\beta-\boldb)$ are identical.

\subsection{The toric category}

There is a bijection between faces $\tau$ of the cone $\RR_{\geq 0} A$ and
$A$-graded prime ideals $I_A^\tau=I_A+R_A\{\del_j\mid j\not\in\tau\}$
of $R_A$ containing $I_A$. If the origin is a face of $\RR_{\geq 0}A$, it
corresponds to the ideal $I^\emptyset_A=(\del_0,\ldots,\del_n)$. In
general, $R_A/I_A^\tau\cong \CC[\NN\tau]$.

An $R_A$-module $N$ is
  \emph{toric} if it is $A$-graded and has a (finite)
  $A$-graded composition chain
\[
0=N_0\subsetneq N_1\subsetneq N_2\cdots\subsetneq N_k=N
\]
such that each composition factor $N_i/N_{i-1}$ is isomorphic as
$A$-graded $R_A$-module to an $A$-graded shift
$(R_A/I_A^\tau)(\boldb)$ for some $\boldb \in\ZZ A$ and some face
$\tau$. The category of toric modules is closed under the formation of
subquotients and extensions.

For toric input $N$, the modules $\calH_\bullet(N;\beta)$ are
holonomic. As $D_A$ is $R_A$-free, any short exact sequence $0\to
N'\to N\to N''\to 0$ of $A$-graded $R_A$-modules produces a long exact
sequence of Euler--Koszul homology. If $\beta$ is not a quasi-degree
of $N$ then the complex
$\calK_\bullet(N;\beta)$ is exact, and if $N$ is a maximal
Cohen--Macaulay module then $\calK_\bullet(N;\beta)$ is a 
a resolution of
$\calH_0(N;\beta)$.

\comment{Let $N=\CC(-\alpha)$ be the graded $R_A$-module whose module structure
is that of $R_A/I_A^\emptyset=R_A/(\del_1,\ldots\del_n)\cong \CC$, and
which lives entirely inside degree $\alpha\in\ZZ^d$.
If $\beta\neq\alpha$,
$\calK_\bullet(N;\beta)$ is an exact complex while
its differentials are zero if $\beta=\alpha$.
The local cohomology groups $H^\bullet_{\del_A}(S_A)$ are the
cohomology groups of the \v Cech complex on $R_A$ induced by $\del_A$,
tensored with $S_A$. Since all modules in the \v Cech complex are
flat, $H^i_{\del_A}(S_A)=\Tor_{n-i}^{R_A}(H^n_{\del_A}(R_A),S_A)$.
}

\subsection{The Euler space}

\begin{ntn}\label{ntn-Euler-space}
The $\CC$-linear span of the Euler operators $\{E_i\}_1^d$ is called the
\emph{Euler space}. Let $E$ be in the Euler space. Then $E$ is in a
unique fashion (as $\rk(A)=d$) a
linear
combination $E=\sum c_i E_i$.
With
$\beta_E:=\sum c_i\beta_i$ we have $E-\beta_E\in
H_A(\beta)$. We further write $\deg_E(-)$ for the degree function
$\sum c_i\deg_{A,i}(-)$.
\end{ntn}
Denote $\theta_j=x_j\del_j$ and $\theta=(\theta_0,\ldots,\theta_n)$.
A linear combination $\sum_j v_j \theta_j$ is in the Euler space if
and only if the coefficient vector $\boldv =(v_0,\ldots,v_n)$,
interpreted as a linear functional on $\CC^{n+1}$ via
$\boldv((q_0,\ldots,q_n)):=\sum v_iq_i$, is the pull-back via $A$ of a
linear functional on $\CC^d$. In other words,
\[
[\boldv\cdot\theta^T=\sum_j v_j\theta_j\text{ is in the Euler
    space}]\Leftrightarrow [\boldv=\boldc\cdot A\text{ for some
  }\boldc\in \CC^d].
\]
If $L\colon \CC^{d}\to\CC$ is a linear functional then the Euler
operator in $H_A(\beta)$ corresponding to its image under
$\Hom_\CC(\CC^d,\CC)\stackrel{\cdot A}{\to}\Hom_\CC(\CC^{n+1},\CC)$ 
is denoted $E_L-\beta_L$.

\begin{lem}
\label{lem-Euler}
For any set $F$ of columns of $A$ contained in a hyperplane that
passes through the origin of $\CC^d$ but does not contain $\bolda_k$,
there is an Euler operator $E_F-\beta_F$ in $H_A(\beta)$ such that the
coefficient of $\theta_j$ in $E_F$ is zero for all $j\in F$, and equal
to $1$ for $j=k$. If $\RR_{\geq 0} F$ is a facet of $\RR_{\geq 0} A$ then
$E_F-\beta_F$ is unique.
\end{lem}
\begin{proof}
Choose for any such set $F$ a linear
functional $L\colon \QQ^d\to\QQ$ that
 vanishes on $F$ while $L(\bolda_k)=1$.
The corresponding Euler operator
$E_L-\beta_L$ has the desired properties, and if we define
numbers $a_{L,j}$ by
\[
E_L=:\sum_j a_{L,j}x_j\del_j
\]
then $a_{L,j}=L(\bolda_j)$. The uniqueness in the facet case is obvious.
\end{proof}



\section{Restricting the Fourier transform}

The Fourier transform $\calF(-)$ is a functor from the category of
$D$-modules on $V$ to the category of $D$-modules on the dual space
$V^*=\Hom_\CC(V,\CC)$. 
In this section we bound the $b$-function along a
coordinate hyperplane 
of the Fourier transform $\calF(M_A(\beta))$ of the
hypergeometric system. Note that this module is called $\check{
M}^\beta_A$ in \cite{ReicheltSevenheck-15}. 

The square of the Fourier transform is the involution induced by
$x\mapsto -x$, which has no effect on the analytic properties of the
modules we study. In particular, $b$-functions along coordinate
hyperplanes are unaffected by this
involution and we therefore consider $\calF^{-1}(M_A(\beta))$
without harm.

We start with introducing some notation.
\begin{ntn}
Let $\{y_j\}$ be the coordinates on $V^*$ such that
$\calF^{-1}(\del_j)=y_j$ on the level of differential operators.  We
let $\tilde D_A$ be the ring of $\CC$-linear differential operators on
$\tilde O_A:=\CC[y_0,\ldots,y_n]$, generated by $\{y_j,\delta_j\}_0^n$
where $\delta_j$ denotes $\frac{\del}{\del y_j}$. Then
$\calF^{-1}(x_j)=-\delta_j$. The subring
$\CC[\delta_1,\ldots,\delta_n]$ of $\tilde D_A$ is denoted $\tilde R_A$. The
isomorphism $(\tilde{-})\colon D_A\to\tilde D_A$ induced by
$\tilde\del_j:=y_j$ and $\tilde x_j=\delta_j$ sends $O_A$ to $\tilde
R_A$ and $R_A$ to $\tilde O_A$.

Thus, $\tilde I_A:=\calF^{-1}(I_A)$ is an ideal of $\tilde O_A$; the
advantage of considering $\calF^{-1}$ rather than $\calF$ is that
$\tilde I_A$ retains the shape of the generators of $I_A$ as
differences of monomials. For each $j$ set 
$\tilde\theta_j:=\calF^{-1}(\theta_j)=-\delta_j y_j$.
The $i$-th level $V$-filtration on $\tilde D_A$ along $y_t$ is 
spanned by $\delta^\alpha y^\beta$ with $\beta_t-\alpha_t\geq i$.
\end{ntn}

Before we get into the technical part, let us show by example an
outline of what is to happen. 
\begin{exa}
Let $A=\begin{pmatrix}-1&0&1\\1&1&1\end{pmatrix}$, a matrix whose
associated semigroup ring is a normal complete intersection. We will
estimate the $b$-function for restriction to the hyperplane $y_1=0$
(corresponding to the middle column) of $\calF^{-1}(M_A(\beta))$.
\begin{figure}[h]
\caption{Restriction of the Fourier transform to $y_1=0$.}
\begin{center}
\includegraphics[width=0.4\textwidth]{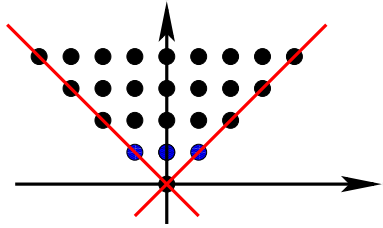}%
\end{center}
\end{figure}

The ideal $\tilde H_A(\beta):=\calF^{-1}(H_A(\beta))$ is generated by
\begin{gather}
-\tilde \theta_0+\tilde \theta_2-\beta_1, \qquad
\tilde \theta_0+\tilde \theta_1+\tilde \theta_2-\beta_2,
\qquad y_0y_2-y_1^2.
\end{gather}
Since $y_1\in (V^1\tilde D_A)$, $y_0y_2$ and hence also
$\tilde\theta_0\tilde\theta_2$ are in $(V^1\tilde D_A)+\tilde
H_A(\beta)$. 
The strategy of the example, and of the theorem in this
section, is to multiply the element $1\in\tilde D_A$ by suitable Euler
operators so that the result is a sum of a polynomial $p(\tilde
\theta_1)$ with an element of
$\CC[\tilde\theta_0,\tilde\theta_1,\tilde\theta_2]\cdot
\tilde\theta_0\tilde\theta_2$; this certifies $p(\tilde\theta_1)$ to
be in $\tilde H_A(\beta)+(V^1\tilde D_A)$.

In the case at hand, the relevant Euler operators are $2\tilde
\theta_0+\tilde\theta_1+\beta_1-\beta_2$ 
and $\tilde\theta_1+2\tilde \theta_2-\beta_1-\beta_2$.
Modulo $\tilde H_A(\beta)$ we can rewrite $(V^1\tilde D_A) \ni
4\delta_0\delta_2y_1^2\equiv 
4\tilde\theta_0\tilde\theta_2\equiv (-\tilde\theta_1-\beta_1+\beta_2)
(-\tilde\theta_1+\beta_1+\beta_2)$. It follows that $(\tilde
s+\beta_1-\beta_2)(\tilde s-\beta_1-\beta_2)$ is a multiple of the
$b$-function, where $\tilde s=\tilde\theta_1=-y_1\delta_1-1$.  
This Fourier twist in
the argument of the $b$-function occurs naturally throughout and we
will make our computations in this section in terms of $b(\tilde s)$.

The expressions $\tilde\theta_1+2\tilde\theta_2$ and
$2\tilde\theta_0+\tilde\theta_1$ that appear in the Euler operators we
used can be found systematically as follows. Let $d_1,d_2$ denote the
coordinates on the degree group $\ZZ^2$ corresponding to $E_1$ and
$E_2$; compare the discussion following Notation
\ref{ntn-Euler-space}. An element of $S_A$ has degree on the facet
$\RR_{\geq 0} \bolda_0$ if and only if the functional
$L_1(d_1,d_2)=d_1+d_2$ vanishes, and the Euler field that corresponds
to this functional in the spirit of Lemma~\ref{lem-Euler} is exactly
$\theta_1+2\theta_2-\beta_1-\beta_2$. The elements of $S_A$ with
degree on the facet $\RR_{\geq 0}\bolda_2$ are determined by the
vanishing of $L_2(d_1,d_2)=d_2-d_1$ and the Euler field corresponding
to this functional is exactly $2\theta_0+\theta_1+\beta_1-\beta_2$. It
is no coincidence that the union of the kernels of these two
functionals is exactly the set of quasi-degrees of $S_A/\del_1\cdot
S_A$. The point is that modulo $\tilde H_A(\beta)$ all monomials in
$\tilde S_A$ with degree in $\RR_+A$ are already in $(V^1\tilde
D_A)$. The task is then to deal with those with degree on the boundary
through multiplication with suitable expressions.

The picture shows in blue the elements of $A$, in black the other
elements of $\NN A$, and in red the quasi-degrees of $S_A/\del_1\cdot
S_A$. Note finally that $(\beta_2-\beta_1)\bolda_1$ and
$(\beta_1+\beta_2)\bolda_1$ are the intersections of $\RR\cdot
\bolda_1$ with $\qdeg_A(S_A)+\beta$.
\end{exa}

We now generalize the computation of the example to the general
case. 
\begin{cnv}
For the remainder of this section we consider restriction to the
hyperplane $y_0$ in order to save overhead (in terms of a further
index variable).
\end{cnv}

Consider the toric module $N=S_A/\del_0S_A$, and take a toric
filtration
\[
(N)\qquad\qquad\qquad\qquad\qquad\qquad\qquad\qquad\qquad 0=N_0\subsetneq N_1\subsetneq \ldots\subsetneq N_k=N\qquad\qquad\qquad\qquad\qquad\qquad\qquad
\]
with composition factors
\[
\bar N_\alpha:=N_\alpha/N_{\alpha-1},
\]
each isomorphic to some shifted face ring
$S_{F'_\alpha}(\boldb_\alpha)$, $F'_\alpha=\tau_\alpha\cap A$,
attached to a face $\tau_\alpha$ of $\RR_{\geq 0} A$.  (We will call
such $F'_\alpha$ also a face.) Lifting the
$N_\alpha$ to $S_A$ yields an increasing sequence of $A$-graded ideals
$J_\alpha\ni\del_0$ of $S_A$ with $N_\alpha=J_\alpha/\del_0\cdot S_A$.

Choose for each composition factor a facet $F_\alpha$ containing
$F'_\alpha$.  Note that none of the faces $F'_\alpha$ will contain
$\bolda_0$ (as $\del_0$ is zero on $N$ but not nilpotent on any
face ring of a face containing $\bolda_0$) and hence we can arrange
that the corresponding facets do not contain $\bolda_0$ either.

Lemma \ref{lem-Euler} produces for each $\bar N_\alpha$ a facet
$F_\alpha$ and corresponding functional $L_{F_\alpha}$ (which we
abbreviate to $L_\alpha$) that vanishes on
the facet and evaluates to $1$ on $\bolda_0$. The associated Euler
operator in $H_A(\beta)$ is $E_{F_\alpha}-\beta_{F_\alpha}$.
Since $L_\alpha$ is zero on
all $A$-columns in $F_\alpha$ and since $\bar N_\alpha$ is a shifted
quotient of $S_{F_\alpha}$, there is a unique value for $L_\alpha$ on
the $A$-degrees of all nonzero $A$-homogeneous elements of $\bar
N_\alpha$. We denote this value by $L_\alpha(\bar N_\alpha)$. Note,
however, that $L_\alpha(\bar N_\alpha)$ does very much depend on the
choice of the facet $F_\alpha$ even though the notation does not
remember this.

Now let $T_\alpha$ be the image in $\calF^{-1}(M_A(\beta))$ of
$\calF^{-1}(J_\alpha)$ under the map induced by $\tilde
O_A\to \tilde D_A\to \calF^{-1}(M_A(\beta))$.  Note
that the image of $T_0=y_0\tilde O_A$ in $\calF^{-1}(M_A(\beta))$ is in
$(V^1\tilde D_A)\cdot \bar 1 $, the bar denoting cosets in
$\calF^{-1}(M_A(\beta))$.

\begin{lem}
\label{lem-kappa}
In the context above, let $\kappa_\alpha$ be the constant
$L_\alpha(\bar N_\alpha)$.
Then in $\calF^{-1}(M_A(\beta))$, modulo the image of $(V^1\tilde D_A)$,
\[
(\tilde \theta_0+\kappa_\alpha-\beta_\alpha)\cdot
(V^0\tilde D_A)\cdot {T_\alpha} = (V^0\tilde D_A)\cdot (\tilde
\theta_0+\kappa_\alpha-\beta_\alpha)\cdot
{T_\alpha} \subseteq (V^0\tilde D_A)\cdot {T_{\alpha-1}}.
\]
\end{lem}

\begin{proof}
Since the commutators $[\tilde \theta_0,(V^0\tilde D_A)]$ are in
$(V^1\tilde D_A)$, it
suffices to show that $(\tilde
\theta_0+\kappa_\alpha-\beta_\alpha)\cdot T_\alpha\subseteq
(V^0\tilde D_A)\cdot T_{\alpha-1}$ modulo $\calF^{-1}(H_A(\beta))$.

By definition, $\tilde
E_\alpha-\beta_\alpha:=\calF^{-1}(E_\alpha-\beta_\alpha)$ is zero in
$\calF^{-1}(M_A(\beta))$.  Take a monomial $\tilde m\in
\tilde O_A$ whose coset lies in $T_\alpha\setminus T_{\alpha-1}$.  By Equation
\eqref{eq-deg}, $\tilde E_\alpha\cdot \tilde m=\tilde m(\tilde
E_\alpha-\kappa_\alpha)$ since $\calF^{-1}(-)$ is a homomorphism.
  Now write $E_\alpha=\sum
a_{\alpha,j}\theta_j$; as before we have
$a_{\alpha,j}=L_\alpha(\bolda_j)$.

Since the coefficient of $\theta_0$ in $E_\alpha$ is $1$, it follows
that in $\calF^{-1}(M_A(\beta))$:
\begin{eqnarray*}
\tilde\theta_0  {\tilde m}&=&(-\tilde E_\alpha+\tilde\theta_0)
            {\tilde m}+\tilde E_\alpha  {\tilde
              m}\\ &=&\sum_{j\neq 0\atop L_\alpha(\bolda_j)\neq
              0}a_{\alpha,j}\delta_jy_j {\tilde
              m}+{\tilde m(\tilde
              E_\alpha-\kappa_\alpha)}\\ &=&\sum_{j\neq 0\atop
              \bolda_j\not\in F_\alpha}a_{\alpha,j}\delta_jy_j
            {\tilde m}+ {\tilde m}(\beta_\alpha-\kappa_\alpha).
\end{eqnarray*}
Recall that $F_\alpha$ contains $F'_\alpha$ and that $\bar N_\alpha$
is a $\ZZ A$-shift of $S_{F'_\alpha}=R_A/I_A^\tau$, whence each $y_j$ with
$\bolda_j\not\in F'$ annihilates $\calF^{-1}(\bar
N_\alpha)$. Therefore, each term $a_{\alpha,j}\delta_j(y_jm)$ in the
last sum of the display is in $(V^0D_A)T_{\alpha-1}$. It follows that
in $\calF^{-1}(M_A(\beta))$ we have 
$(\tilde\theta_0+\kappa_\alpha-\beta_\alpha){T_\alpha}\subseteq
(V^0\tilde D_A){T_{\alpha-1}}$ as claimed.
\end{proof}

\begin{thm}\label{thm-fourier}
For $t=0,\ldots,n$, the number $\eps\in\CC$ is a root of the
$b$-function $b(\tilde s)$ (with $\tilde s=\tilde\theta_t=-\delta_ty_t$) of
$\calF^{-1}(M_A(\beta))$ along $y_t=0$, only if $\eps\cdot \bolda_t$
is a point of intersection of the line $\CC\cdot \bolda_t$ with
the set $\beta-\qdeg_A(N)$, the quasi-degrees of the toric module
$N=S_A/\del_tS_A$ multiplied by $-1$ and shifted by $\beta$.

\end{thm}
\begin{proof}
Without loss of generality we shall suppose that $t=0$ by way of re-indexing.

We will show that a divisor of  
$\prod_\alpha(\tilde \theta_0+\kappa_\alpha-\beta_\alpha)$ is inside
$H_A(\beta)+(V^1\tilde D_A)$, in notation from the
previous lemma.

Indeed, it follows from Lemma \ref{lem-kappa} that
$\prod_\alpha(\tilde\theta_0+\kappa_\alpha-\beta_\alpha)$ multiplies
$\bar 1\in \calF^{-1}(M_A(\beta))$ into $(V^0\tilde D_A)\cdot y_0\cdot \bar
1\subseteq (V^1\tilde D_A)\cdot \bar 1$. Hence the root set of the
$b$-function $b(\tilde\theta_0)$ in question is a subset of
$\{\beta_\alpha-\kappa_\alpha\}$, $\alpha$ running through the indices
of the chosen composition series of $N$. This set is determined by the
composition series $(N)$ and the choices of the facets $F_\alpha$ for
each $N_\alpha$.  Varying over all choices of facets $\{F_\alpha\}$
for a given chain $(N)$, the root set of $b(\tilde\theta_0)$ is in the
intersection $\rho_N$ of all possible sets
$\{\beta_\alpha-\kappa_\alpha\}_{\alpha\in(N)}$.

Since $L_\alpha(\bolda_0)=1$, the point
$(\beta_\alpha-\kappa_\alpha)\cdot \bolda_0$ is the intersection of
the hyperplane $L_\alpha=\beta_\alpha-\kappa_\alpha$ with the line
$\CC\cdot\bolda_0$.  Thus, $\rho_N$ is inside the intersection of
$\CC\cdot \bolda_0$ with all arrangements
$\Var\prod_\alpha(L_\alpha-\beta_\alpha+\kappa_\alpha)$.  The
intersection of the arrangements
$\Var\prod_\alpha(L_\alpha-\beta_\alpha+\kappa_\alpha)$ is the union
of the quasi-degrees of all $\bar N_\alpha$ of the composition chain $(N)$,
multiplied by $-1$ and shifted by $-\beta_\alpha$. As $N$ is finitely generated,
$\qdeg_A(N)=\bigcup_\alpha \qdeg_A(\bar N_\alpha)$. Hence the root set
of $b(\tilde \theta_0)$ is contained in the intersection
$-\qdeg_A(S_A/\del_0S_A)+\beta$ with $\CC\cdot\bolda_0$.

\end{proof}


\begin{rmk}
The quantity $\tilde\theta_t$ is the more natural argument for the
$b$-function here. Note that the roots of $b(y_t\delta_t)$ are those
of $b(\tilde\theta_t)$ shifted up by $1$ and then multiplied by
$-1$. 
\end{rmk}

\begin{exa}
Let
$A=(\bolda_0,\bolda_1,\bolda_2)=\begin{pmatrix}-1&0&3\\1&1&1\end{pmatrix}$
and $\beta={\beta_1\choose \beta_2}$. The ring $S_A$ is a complete intersection
but not normal. 

Consider restriction to $y_1=0$ (the middle column).  Then
$N=S_A/\del_1\cdot S_A$ has a toric filtration involving 4 steps,
given by the ideals $0\subsetneq \del_0^3\cdot N\subsetneq \del_0^2\cdot
N\subsetneq \del_0\cdot N\subsetneq N$. The 
corresponding $A$-graded composition factors are
$S_A(-3\cdot\bolda_0)/(\del_1,\del_2)S_A$ and $\{
S_A(-\alpha\cdot\bolda_0)/(\del_0,\del_1)S_A\}_{\alpha=0}^2$.
The $b$-function $b(\tilde\theta_1)$ for the inverse
Fourier transform is $(\tilde\theta_1-\beta_1-\beta_2)
\prod_{\alpha=0}^2(\tilde\theta_1-\frac{3\beta_2-\beta_1-4\alpha}{3})$.

Explicitly, $y_1^4-y_0^3y_2\in \tilde H_A(\beta)$ gives $(V^1\tilde
D_A)\ni\delta_0^3\delta_2y_0^3y_2=\tilde\theta_2\tilde\theta_0(\tilde\theta_0-1)(\tilde\theta_0-2)$
which modulo $\tilde H_A(\beta)$ equals $(-1)^4(\tilde\theta_1-\beta_1-\beta_2)
\prod_{\alpha=0}^2(\tilde\theta_1-\frac{3\beta_2-\beta_1-4\alpha}{3})$. 
The relevant Euler operators are $\theta_1+4\theta_2-\beta_1-\beta_2$
and $3\theta_1+4\theta_0-3\beta_2+\beta_1$.


\begin{figure}[h]
\caption{Restriction of the Fourier transform to $y_1=0$.}
\begin{center}
\includegraphics[width=0.56\textwidth]{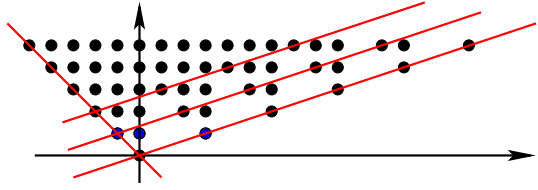}%
\end{center}
\end{figure}

The picture shows in blue the columns of $A$, in black the other
elements of $\NN A$, in red the quasi-degrees of $N=S_A/\del_1\cdot
S_A$. The roots of $b(\delta_1y_1)$ (which are opposite to the roots
of $b(\tilde \theta_1)$) are the intersections of the line
$\CC\cdot{0\choose 1}$ with the shift of the red lines
by $-\beta$.

In this example, each composition factor
corresponds to facet and to a component of the quasi-degrees of $N$.
One checks that each composition chain must have these four lines as
quasi-degrees. Note, however, that composition chains are far from
unique and in general such correspondence will not exist. 
\end{exa}

\begin{rmk}
The $b$-function for $\calF^{-1}(M_A(\beta))$ along a coordinate
hyperplane  is generally
not reduced, and its degree may be lower than the length of the
shortest toric filtration for $N=S_A/\del_t\cdot S_A$ would suggest.
(Not every component of $\beta-\qdeg_A(N)$ needs to meet the
line $\CC\cdot \bolda_t$).
\end{rmk}

\begin{cor}\label{cor-fourier-roots-nice}
The roots of the $b$-function $b(\delta_ty_t)$ of $\calF^{-1}(M_A(\beta))$
along $y_t=0$ are in the field
$\QQ(\beta)$.

Consider $\calF^{-1}(M_A(0))$; then:
\begin{asparaenum}
\item the roots of the $b$-function $b(\tilde\theta_t)$ are
  non-negative rationals;
\item if $S_A$ is normal, all roots are in the interval $[0,1)$;
\item if the interior ideal of $S_A$ is contained in $\del_t\cdot S_A$
then zero is the only root.
\end{asparaenum}
\end{cor}
\begin{proof}
The first claim is a consequence of the intersection property in
Theorem \ref{thm-fourier}: the defining equations for the quasi-degrees
are rational.

Let $N=S_A/\del_tS_A$. For items 1.-3., we need to study the
intersection of $\qdeg_A(N)$ with $\CC\cdot\bolda_t$, since $\beta=0$
and $\delta_ty_t=-\tilde\theta_t$. 
The quasi-degrees of $N$ are covered by hyperplanes of
the sort $L_\alpha=\varepsilon$ where $L_\alpha$ is a rational
supporting functional of the facet $F_\alpha$. In particular, we
can arrange $L_\alpha$ to be zero on $F_\alpha$, positive on the rest of $A$,
and $L_\alpha(\bolda_t)=1$. As $\deg_A(N)\subseteq \deg_A(S_A)$,
$\eps\geq 0$. Hence $\Var(L_\alpha-\eps)$ meets $\CC\cdot \bolda_t$ in
the non-negative rational multiple $\eps\bolda_t$ of $\bolda_t$.  If
$S_A$ is normal, $\deg_A(S_A/\del_AS_A)$ is covered by hyperplanes
$\Var(L_\alpha-\eps)$ that do not meet the cone $\bolda_t+\RR_{\geq
  0}A$. These are precisely the ones for which $\eps<1$.

If $\del_t\cdot S_A$ contains the interior ideal then
$\deg_A(N)$, and hence $\qdeg_A(N)$, is
inside the supporting hyperplanes of the cone, which meet $\CC
\cdot \bolda_t$ at the origin.
\end{proof}

\begin{rmk}
One special case in which case 3 of Corollary
\ref{cor-fourier-roots-nice} applies is when $S_A$ is Gorenstein
and where further $\del_t$ generates the canonical module. The matrix
$A=(\bolda_0,\ldots,\bolda_3)=
\begin{pmatrix}1&1&1&1\\0&1&3&0\\0&0&0&1\end{pmatrix}$,
with the interior ideal being generated by $\del_1\del_3$, provides an
example that case (3) can occur in a Gorenstein situation without the
boundary of $\NN A$ being saturated. See \cite{SchulzeWalther-Newton}
for a discussion on Cohen--Maculayness of face rings of
Cohen--Macaulay semigroup rings.
\end{rmk}

\section{$b$-functions for the hypergeometric system}

\subsection{Restriction along a hyperplane}

We are here interested in the $b$-function for the hypergeometric
module $M_A(\beta)$ along the hyperplane $x_t=0$. 
As in the previous section, apart from examples, we actually carry out
all computations for $t=0$, in order to have as few variables around
as possible. On the other hand, the natural argument for expressing
the $b$-function will be $s=x_0\del_0$.

\begin{ntn}
With
$A=(\bolda_0,\ldots,\bolda_n)$ and distinguished index $0$, we denote 
$A':=(\bolda_1,\ldots,\bolda_n)$. Via $\NN A'\subseteq \NN A$
we consider $S_{A'}$ as a subring of $S_{A}$. 

For $k\in \NN$ let $\bar J_{A,0;k}\subseteq S_{A'}$ be the vector
space spanned by the monomials $\del^\boldu$ with $u_0=0$ (so that
$\del^\boldu\in S_{A'}$) that satisfy $\del_0^k\cdot \del^\boldu\in
S_{A'}$. We denote $J_{A,0;k}\subseteq R_{A'}$ the preimage of $\bar
J_{A,0;k}$ under the natural surjection $R_{A'}\onto S_{A'}$.
  Put $J_{A,0}=\sum_{k\geq 1} J_{A,0;k}$ and $\bar
J_{A,0}=J_{A,0}/I_{A'}\subseteq S_{A'}$ .
\end{ntn}
Each $\bar J_{A,0;k}$ is a monomial ideal of $S_{A'}$ since
$\del_0^k(\del^\boldv\del^\boldu)=\del^\boldv(\del_0^k\del^\boldu)$. 
Note, however, that $\bar J_{A,0;k}$ need not
be contained in $\bar J_{A,0;k+1}$.
If $\bolda_0\in\RR_{\geq 0}A'$ then some power of $\del_0$ is in $S_{A'}$ and
so $\bar J_{A,0}=S_{A'}$. 

\begin{dfn}
For $\bolda_0\in\RR^d$ outside $\RR_{\geq 0}A'$, a point
$\bolda\in\RR_{\geq 0}A'$ is \emph{$\bolda_0$-visible} if $\bolda +
\lambda\cdot \bolda_0$, $0<\lambda\ll 1$ is outside $\RR_{\geq 0}A'$.
(The idea behind the choice of language is that the observer stands at
the point of projective space given by the line $\RR\bolda_0$.)

By abuse of notation, we say that $\del^\bolda$ is $\bolda_0$-visible
if $\bolda$ is.
\end{dfn}

\begin{lem}\label{lem-bolda}
Assume that $\bolda_0$ is not in the cone $\RR_{\geq 0}A'$. Then the
radical of $J_{A,0}$ is generated by the $\bolda_0$-invisible elements
of $S_{A'}$, and in
consequence the quasi-degrees of $S_{A'}/J_{A,0}$ are a union of
shifted face spans where each face is in its entirety visible from
$\bolda_0$.
\end{lem}
\begin{proof}
If $\ZZ A/\ZZ A'$ has positive rank then all points of $\NN A$ are
$\bolda_0$-visible while $J_{A,0}$ is clearly zero, so that in this
case there is nothing to prove. We therefore assume that $\ZZ A/\ZZ
A'$ is finite.

It is immediate that $\bolda$ is $\bolda_0$-visible if and only if any
positive integer multiple of it is. This implies that no power of an
$\bolda_0$-visible element $\del^\bolda$ of $S_{A'}$ can be in the radical of
$J_{A,0}$ since $\del^{m\cdot\bolda + k\bolda_0}$ can't have its
degree in the cone of $A'$.

For the converse, suppose $\bolda$ is not $\bolda_0$-visible, so that
there are positive integers $p<q$ with
$\bolda+(p/q)\cdot\bolda_0\in\RR_{\geq 0}A'$. Then a high power of
$\del^{q\cdot \bolda+p\cdot\bolda_0}$ is in $\CC[\ZZ A\cap \RR_{\geq
    0}A']$ and a suitable power $\del^\boldb$ of that will be in
$\CC[\ZZ A'\cap \RR_{\geq 0}A']$ because of the finiteness of $\ZZ
A/\ZZ A'$. Now let $\tau$ be the smallest face of $\RR_{\geq 0}A'$
that contains $\boldb$; this makes $\boldb$ an interior point of
$\tau$. Since $\CC[\tau\cap \ZZ A']$ is a finitely generated
$\CC[\tau\cap\NN A']$-module, some power of $\del^\boldb$ is in
$\CC[\tau\cap\NN A']\subseteq S_{A'}$. This shows that some power of
$\del^{q\cdot \bolda}$ times some power of $\del^{p\cdot\bolda_0}$ is
in $S_{A'}$, establishing the first claim of the lemma.

In every composition chain for $S_{A'}/J_{A,0}$, each
composition factor is an $S_{A'}/\sqrt{J_{A,0}}$-module. Thus the
quasi-degrees of $S_{A'}/J_{A,0}$ are inside a union of shifted
quasi-degrees of $S_{A'}/\sqrt{J_{A,0}}$ and hence all
$\bolda_0$-visible, which implies the second claim.
\end{proof}

Our main theorem in this section is:

\begin{thm}\label{thm-gkz}
The root locus of the $b$-function $b(x_0\del_0)$ for restriction of
$M_{A}(\beta)$ along $x_0=0$ is, up to inclusion of non-negative
integers, contained in the locus of intersection
$(-\qdeg_{A'}(S_{A'}/\bar J_{A,0})+\beta)\cap \CC\cdot\bolda_0$. The
set of integers needed can be taken to be the integers $0,\ldots,k-1$
such that $J_{A,0}= \sum_{1\le i\leq k}J_{A,0;i}$.

In two extreme cases one can be explicit:
\begin{enumerate}
\item if $\dim S_A-1=\dim S_{A'}$ then the $b$-function is linear with
  root given by the intersection of $(-\qdeg_{A}(S_{A'})+\beta)\cap
  \CC\cdot\bolda_0$;
\item if $\bolda_0\in\RR_{\geq 0}{A'}$ then the $b$-function has
  integer roots in $\{0,1,\ldots,k-1\}$ where $k=\min\{t\in\NN\mid
  0\not= t\cdot\bolda_0\in\NN A'\}$.
\end{enumerate}
\end{thm}
\begin{proof}
We first dispose of the extreme cases.  If $\dim S_A-1=\dim S_{A'}$,
then $S_{A}$ is the polynomial ring $S_{A'}[\del_0]$ and $A'$ is a
facet of $A$. By Lemma \ref{lem-Euler} there is
$\boldv=(v_1,\ldots,v_d)$ such that the Euler operator
\[
E-\beta_E=\sum v_i(E_i-\beta_i)
\]
is in $H_A(\beta)$ and equals $\theta_0-\beta_E$.  In particular, the
$b$-function is $s-\beta_E$. On the other hand: $\bar J_{A,0}$ is zero
in this case, $\boldv=(v_1,\ldots,v_d)$ is in the kernel of ${A'}^T$,
and $\bolda_0^T\boldv=1$. Therefore, the quasi-degrees of $S_{A'}/\bar
J_{A,0}$ form the hyperplane given as the kernel of $\boldv$ and
$(\boldv^T\beta)\bolda_0=\beta_E\bolda_0$ is the intersection of
$-\qdeg_{A}(S_{A'})+\beta$ with $\CC\bolda_0$.

If $\bolda_0\in\RR_{\geq 0}A'$ then $\NN\bolda_0$ meets $\NN A'$ and
so $\del_0^k=\del^\boldu$ with $\boldu=(0,u_1,\ldots,u_n)\in\NN
A'$. In particular, $J_{A,0}=S_{A'}$ in this case.  Moreover,
$(x_0\del_0)(x_0\del_0-1)\cdots (x_0\del_0-k+1)=x_0^k\del_0^k=
x_0^k(\del_0^k-\del^\boldu)+x_0^k\del^\boldu \in
H_{A}(\beta)+V^1(D_{A})$ shows the claim made in this case.

Now suppose that $A$ and $A'$ have equal rank but $\bolda_0\not\in
\RR_{\geq 0}A'$. In that case, $\bar J_{A,0}$ is a non-trivial ideal of
$S_{A'}$. We shall use a toric filtration 
\[
(N)\quad :\quad 0=N_0\subsetneq N_1\subsetneq
\ldots\subsetneq N_t=S_{A'}/\bar J_{A,0}
\]
and let $J_\alpha\supseteq J_{A,0}$ be the $R_{A'}$-ideal such that
$N_\alpha=J_\alpha/J_{A,0}$.  We will view $J_\alpha$ as subset of
$D_{A'}$ or even $D_{A}$. In analogy to the previous case, for any
$\del^\boldu$ in $J_{A,0;k}$ the $b$-function along $x_0$ of the coset
of $\del^\boldu$ in $M_{A}(\beta)$ divides
$s(s-1)\cdots(s-k+1)$. Indeed, $\del^\boldu\in J_{A,0;k}$ implies that
$\del_0^k\del^\boldu-\del^\boldv\in I_{A}$ for some $\boldv$ with
$v_0=0$, and so $x_0^k\del_0^k\del^\boldu\in
H_{A}(\beta)+V^{1}(D_{A})$. In particular, the root set of the
$b$-function of the coset of $\del^\boldu$ in $M_{A'}(\beta)$ is
inside the set of integers described in the statement of the theorem.

For each composition factor $\bar N_\alpha=N_\alpha/N_{\alpha-1}$
choose now a facet $\tau_\alpha$ of $A'$ and an element
$\del^{\boldu_\alpha}$ of $S_{A'}$
$\boldu_\alpha
\in\{0\}\times\NN^n$ such that $N_\alpha$ is a quotient of
$S_{A'}\cdot \del^{\boldu_\alpha}$ and such that the annihilator of
$\del^{\boldu_\alpha}$ in $\bar N_\alpha$ contains the toric ideal
$I_{A'}^{\tau_\alpha}$. Then $\qdeg_{A'}(\bar N_\alpha)$ is contained
in $A'\cdot\boldu_\alpha+\qdeg_{A'}(S_{\tau_\alpha})$.

Since $\bolda_0$ is not in $\RR_{\geq 0}A'$, Lemma \ref{lem-bolda}
shows that the facet $\tau_\alpha$ can be chosen such that
$\bolda_0\not\in \QQ\cdot \tau_\alpha$. Indeed, if an entire face of $\RR_{\geq
  0}A'$ is visible from $\bolda_0$ then it sits in at least one facet
whose span does not contain $\bolda_0$.  By Lemma \ref{lem-Euler}
there is an element $E_\alpha$ of the Euler space of $A$ that does not
involve any element of $\tau_\alpha$, but which has coefficient $1$
for $\theta_0$. Notation \ref{ntn-Euler-space} then associates a degree function
$\deg_{E_\alpha}(-)$ to $\alpha$. 

As $\del_j\cdot\del^{\boldu_\alpha}\in N_{\alpha-1}$ for
$j\not\in\tau_\alpha$ it follows that the difference of
$(E_\alpha-\beta_\alpha)\cdot \del^{\boldu_\alpha}$ and
$(\theta_0-\beta_\alpha)\cdot \del^{\boldu_\alpha}$ is inside
$(V^0D_A)N_{\alpha-1}$.  Since $E_\alpha-\beta_\alpha$ is in
$H_{A}(\beta)$, so is $\del^{\boldu_\alpha}(E_\alpha-\beta_\alpha)=
(E_\alpha-\beta_\alpha+\deg_{E_\alpha}(\del^{\boldu_\alpha}))\del^{\boldu_\alpha}$.
Therefore,
$(\theta_0-\beta_\alpha+\deg_{E_\alpha}(\del^{\boldu_\alpha}))\del^{\boldu_\alpha}$
is in $H_{A}(\beta)+(V^0D_{A})N_{\alpha-1}$.  Then, in parallel to how
Lemma \ref{lem-kappa} was used in the proof of Theorem
\ref{thm-fourier}, the product
\[
\prod_\alpha (
\theta_0-\beta_\alpha+\deg_{E_\alpha}(\del^{\boldu_\alpha}))
\]
multiplies $1\in D_A$ into
$H_{A}(\beta)+(V^0D_{A})J_{A,0}+(V^{1}D_A)$. Multiplying by
$x_0^k\del_0^k$ for suitable $k$ one obtains the desired bound for the
$b$-function as in the second paragraph of the proof.

It follows as in Theorem \ref{thm-fourier} (with the
modification that we have here $\theta_0$ rather than
$\calF^{-1}(\theta_0)$, which affects signs) that the intersection of the
roots of all such bounds is the intersection of
$(-\qdeg_{A'}(S_{A'}/\bar J_{A,0})+\beta)$ with the line $\CC\cdot\bolda_0$.
\comment{Finally, note that the $b$-function should also annihilate the
Euler--Koszul homology module $\calH_0(\bar N_\alpha;\beta)$ since
$\del_A$ commutes with $\theta_0$. But each such module requires at
least the indicated linear factor and that shows that the bound is
sharp up to radical.}
\end{proof}

\begin{exa}
With
$A=(\bolda_0,\bolda_1,\bolda_2)=\begin{pmatrix}-1&0&3\\1&1&1\end{pmatrix}$,
consider the $b$-function along $x_1$ of the $A$-hypergeometric
system. The ideal $J_{A,1}$ is generated by $1\in
S_{A'}=\CC[\NN(\bolda_0,\bolda_2)]$ since $\del_1^4$ is in $S_{A'}$. The set of
necessary integer roots is then $\{0,1,2,3\}$. No other roots are
needed since $S_A/J_{A,1}$ is zero, irrespective of $\beta$.

\begin{figure}[h]
\caption{The elements of $S_A\smallsetminus S_{A'}$ (black) and
  $S_{A'}$ (green) for restriction to $x_1$}
\begin{center}
\includegraphics[width=0.5\textwidth]{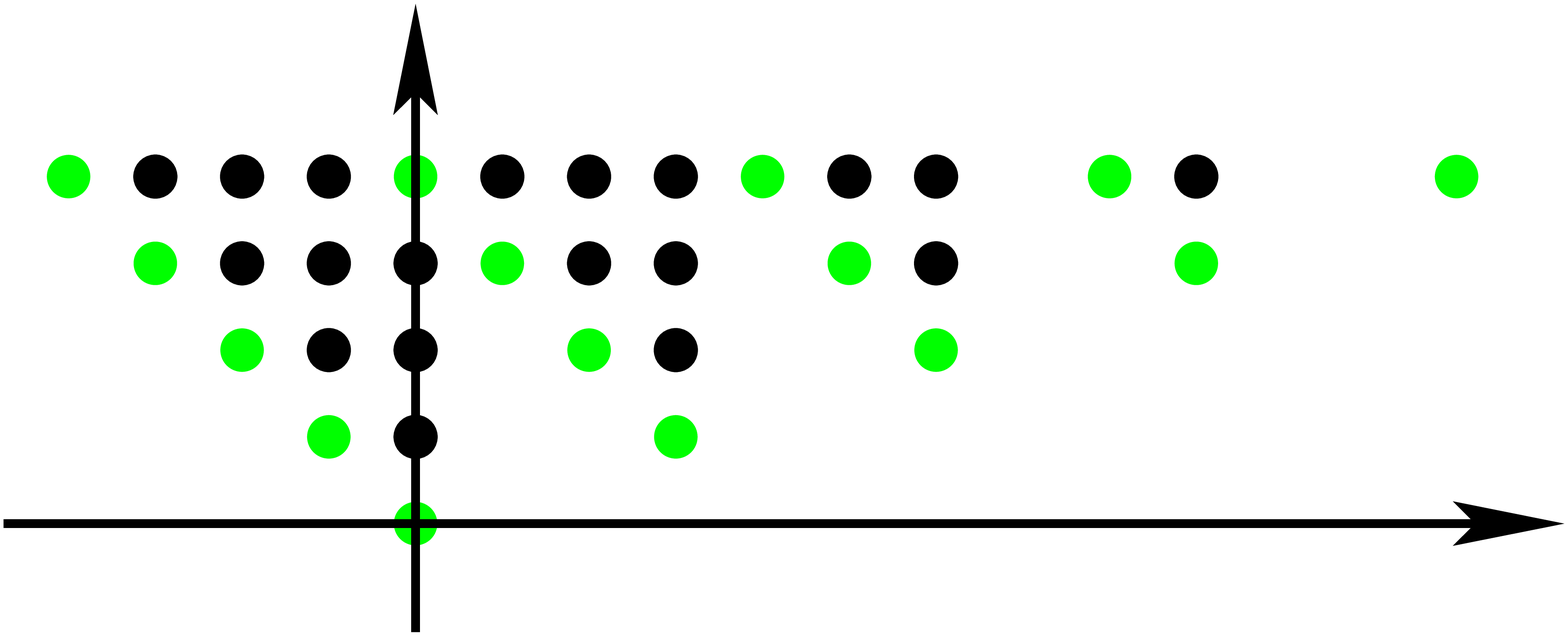}%
\end{center}
\end{figure}

Restriction to $(x_2=0)$ behaves differently. As
$S_{A'}=\CC[\NN(\bolda_0,\bolda_1)]$ now, $J_{A,2}=J_{A,2;1}$ is generated
by $\del_0^3$, and the quasi-degrees of $S_{A'}/J_{A,2}$ are the lines
$\CC\cdot (0,1)+(i,0)$ with $i=0,-1,-2$.
\begin{figure}[h]
\caption{The quasi-degrees of $S_A/J_{A,2}$ form three parallel lines.}
\begin{center}
\includegraphics[width=0.5\textwidth]{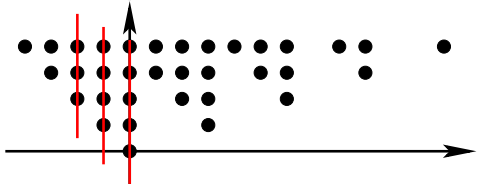}%
\end{center}
\end{figure}
The intersection of the
negative of these
three lines, shifted by $\beta$, with the line $\CC\cdot \bolda_2$ is
$\bolda_2\cdot\{(i+\beta_1)/3\}_{i=0,1,2}$. So the $b$-function has
(at worst) roots $\{0,\beta_1,\beta_1+1,\beta_1+2\}/3$.
\end{exa}
\begin{rmk}
We believe that both bounds in Theorems \ref{thm-fourier} (as is) and
\ref{thm-gkz} (up to integers) are sharp.
\end{rmk}

\subsection{Restriction to a generic point}

We suppose here that $A$ is homogeneous; in other words, the Euler
space contains a homothety.  Let $p=(p_0,\ldots,p_{n})$ be a point
of $\CC^{n+1}$. We wish to estimate here the $b$-function for
restriction of $M_A(\beta)$ to the point $-p$ if $p$ is
generic. As a holonomic module is a connection near any generic point,
this restriction yields a vector space isomorphic to the space of
solutions to $H_A(\beta)$ near $-p$, see \cite[Sec.~5.2]{SST}.

\begin{dfn} \label{dfn-bfu-point} Let
  $\theta_p=(x_0+p_0)\del_0+\ldots+(x_n+p_n)\del_n$ and write
  $\theta$ for $\theta_p$ if $p=0$.
  The $b$-function for restriction of a principal $D$-module $M=D/I$ to the
  point $x+p=0$ is the minimal polynomial $b_p(s)$ such that
  $b_p(\theta_p)\in I+(V^1_p D)$ where
  $V^k_p D$ is the Kashiwara--Malgrange $V$-filtration along
  $\Var(x+p)$:
\[
V^k_p D=\CC\cdot\{(x+p)^\boldu\del^\boldv\quad \mid
\quad |\boldu|-|\boldv|\geq k\}.
\]
\end{dfn}

\begin{rmk}
\begin{asparaenum}
\item For any pair of manifolds $Y\subseteq X$ and given a $D$-module
  $M$ on $X$ one can define a $b$-function of restriction for the
  section $m\in M$ along $Y$ by a formula generalizing both Definition
  \ref{dfn-bfu-hyperplane} and
  Definition \ref{dfn-bfu-point}. Kashiwara proved their existence
  for holonomic $M$.
\item
The roots of this $b$-function here relate to restriction of solution
sheaves as follows. Near a generic point
$x+p=0$, a $D$-module $M$ is a connection whose solution space has
a basis consisting of a certain number of holomorphic functions. The
germs of these functions form a vector space that can be identified
with the dual of the $0$-th homology group of
$(D/(x+p)D)\otimes^L_D M$. Filtering this complex by
$V^\bullet_p D$, $b_p(k)$ annihilates the $k$-th graded part of
its homology, compare \cite{Oaku-AAM,OakuTakayama-algs, Walther-cdrc}. In
particular, $b_p(s)$ carries information on the starting terms of
the solution sheaf of $M$ near $x+p=0$.
\end{asparaenum}
\end{rmk}

The purpose of this section is to bound $b_p(s)$ for $I=H_A(\beta)$
and generic $p$ with the following strategy.  We first show that a
polynomial $b(s)$ is a multiple of $b_p(s)$ if $b(\theta)$ is in
$D_A(I_A,A\cdot\calE\cdot\del)$ where
 \[
\calE=\begin{pmatrix}p_0&0&\cdots&0\\
0&p_1&&\vdots\\
\vdots&&\ddots&0\\
0&\cdots&0&p_n
\end{pmatrix},
\]
provided that $p$ is component-wise nonzero.
The generators of $D_A(I_A,A\cdot\calE\cdot\del)$ are independent of
$x$ and we next observe that the radical of
$R_A(I_A,A\cdot\calE\cdot\del)$ is $R_A\cdot\del$, provided that
$p$ is generic. Thus, $b_p(s)$ will be a factor of any
polynomial that annihilates the finite length module
$R_A/(I_A,A\cdot\calE\cdot\del)$ as long as $p$ is generic. We
exhibit a particular such polynomial with all roots integral. In the
case of a normal semigroup ring, we show that the (necessarily integral)
roots of $b_p(s)$ are in the interval $[0,d-1]$.

\medskip

We begin with pointing out that $b(\theta_p)\in I+(V^1_p D)$ is
equivalent to $b(\theta)\in I_{p}+(V^1_0D)$ where $I_{p}$ is the image
of $I$ under the morphism induced by $x\mapsto x-p$, $\del\mapsto\del$
and $(V^k_0D)$ is the Kashiwara--Malgrange filtration along the
origin.  Among the generators of $I=H_A(\beta)$, only the Euler
operators depend on $x$ while $(I_A)_p=I_A$ for any $p$; one has
$(E_i-\beta_i)_{p}=\sum
a_{i,j}(x_j-p_j)\del_j-\beta_i=E_i-\beta_i-\sum a_{i,j}p_j\del_j$.  We
hence seek a relation $b(\theta)\in D_A\cdot (I_A,E-\beta-A\cdot
\calE\cdot\del)+(V^1_0D_A)$ with $\calE$ as above.

Generally, a statement $b(\theta)\in I+(V^1_0D_A)$ is equivalent to
$b(\theta)$ being in the degree zero part $\gr^0_{V_0}(I)$ of the
associated graded object. Note that $\gr_{V_0}(D_A)$ is a Weyl algebra
again (although of course the symbol map $D_A\to \gr_{V_0}(D_A)$ is
not an isomorphism).  Abusing notation, we denote $x$ and $\del$ also
the symbols in $\gr_{V_0}(D_A)$ of the respective elements of $D_A$.
By the previous paragraph then, the graded ideal
$\gr_{V_0}(H_A(\beta)_p)$ contains the elements that generate $I_A$
(since $I_A$ is homogeneous!), as well as the elements
$A\cdot\calE\cdot\del$ which arise as the $V_0$-symbols of
$E_p-\beta$.

We need the following folklore result ) for which we know no explicit
reference. 
\begin{clm}
The $R_A$-ideal generated by $I_A$ and $A\cdot\calE\cdot\del$ has, for
generic $\calE$, radical $R_A\cdot\del$. 
\end{clm}
A sequence of $d$ generic linear forms is of course a system of
parameters on $S_A$; the issue is to show that linear forms of the
type $A\cdot\calE\cdot\del$ are sufficiently generic.
\begin{proof}
As $I_A$ and $A\cdot\calE\cdot\del$ are standard graded,
$\Var(I_A,A\cdot \calE\cdot\del)$ is a conical
variety. It thus suffices to show that the ideal $\Var(I_A,A\cdot
\calE\cdot\del)$ is of height $n+1$.

The ideal $R_A[x](I_A,A\cdot\theta)$ in the polynomial ring $R_A[x]$
defines in the cotangent bundle $\Spec(R_A[x])$ of $\CC^{n+1}$ the
union of the conormals to each torus orbit since the Euler fields are
tangent to the torus and span a space of the correct dimension in each
orbit point. Suppose the claim is false, so that there is a nonzero
point $y\in\Var(I_A)$ such that (the generically chosen vector) $p$ is
a conormal vector to the orbit of $y$.  If $y$ is in a torus orbit
$O_\tau$ associated to a proper face $\tau$ of $A$ then its
coordinates corresponding to $A\smallsetminus\tau$ are zero and we can
reduce the question to the case where $A=\tau$. It is hence enough to
show that there is $p\in\CC^{n+1}$ such that $p$ is not a conormal
vector to any smooth point of $\Var(I_A)$.

Let $X\subseteq \CC^{n+1}$ be any reduced affine variety and denote
$X_0$ its smooth locus. We define a set $C(X)$ inside $\CC^{n+1}$ by
setting
\[
[\eta\in C(X)]\iff [\exists y\in X_0,\quad \eta\in (T^*_{X_0}(\CC^{n+1}))_y]
\]
where $(T^*_{X_0}(\CC^{n+1}))_y$ is the fiber of the conormal bundle
at $y$ of the pair $X_0\subseteq \CC^{n+1}$. This is a constructible,
analytically
parameterized union of a $\dim(X)$-dimensional family of vector spaces
of dimension $n+1-\dim(X)$, which hence might fill $\CC^{n+1}$.  

Now suppose that $X$ is a conical variety; then the conormals of $y$
and $\lambda y$ agree for all $\lambda\in\CC^*$. In particular,
\[
C(X)=\bigcup_{\bar y\in\Proj(X)}(T^*_{X_0}(\CC^{n+1}))_y
\]
where $\Proj(X)$ is the associated projective variety. But this is now
an analytically parameterized union of a $(\dim(X)-1)$-dimensional
family of vector spaces of dimension $n+1-\dim(X)$.  It follows that
most elements of $\CC^{n+1}$ are outside $C(X)$ in this case, and the
claim follows. 
\end{proof}
It follows from the Claim that $\gr_{V_0}(H_A(\beta)_p)$ contains
all monomials in $\del$ of a certain degree $k$ that depends on
$A$. Let $E=\theta_0+\ldots+\theta_n$; by hypothesis $E-\beta_E\in
H_A(\beta)$. 
\begin{lem}
Denote $\del_A^k$ the set of all monomials of degree $k$ in
$\del_0,\ldots,\del_n$, and $D_A\cdot\del_A^k$ the left $D_A$-ideal
generated by $\del_A^k$. Then in $D_A/D_A\cdot \del_A^k$, the identity
$E(E-1)\cdots(E-k+1)\cong 0$ holds.
\end{lem}
\begin{proof}
This is clear if $k=1$. In general, by induction,
\[
E(E-1)\cdots(E-k+1)\in
D_A\cdot\del_A^{k-1}\cdot(E-k+1)=D_A\cdot E\cdot\del_A^{k-1}\subseteq D_A\cdot
\del_A^k.
\]
\end{proof}

\begin{rmk} 
The homogeneity of $X$ is necessary in the Claim, since otherwise
$C(X)$ does not need to be contained in a hypersurface. Consider, for
example, $A=(2,1)$ in which case the union of all tangent lines
(nearly) fills the plane, and where the zero locus of $I_A$ and
$A\cdot\calE\cdot\del$ contains always at least two points.
\end{rmk}

The lemma implies that $\gr_{V_0}^0(H_A(\beta)_p)$ contains
$E(E-1)\cdots(E-k+1)$ if $p$ is generic. In other words, the
$b$-function for restriction of $M_A(\beta)$ to a generic point
divides $s(s-1)\cdots (s-k+1)$.

In some cases one can be more explicit about $k-1$, the top degree in
which $R_A/R_A(I_A,A\cdot\calE\cdot\del)$ is nonzero. Suppose $S_A$ is
a Cohen--Macaulay ring, then systems of parameters are regular
sequences. In particular, the Hilbert series of
$Q_A:=R_A/R_A(I_A,A\cdot\calE\cdot\del)$ is that of $S_A$ multiplied
by $(1-t)^d$. Suppose in addition, that $S_A$ is normal. Since we
already assume that $S_A$ is standard graded, let $P$ be the polytope
that forms the convex hull of the columns of $A$. The Hilbert series
of $S_A$ is then of the form $\sum_{m=0}^\infty p_m\cdot t^m$ where
$p_m$ is the number of lattice points in the dilated polytope $m\cdot
P$. This number of lattice points is counted by the Erhart polynomial
$E_P(m)$ of $P$, a polynomial of degree $d-1=\dim(P)$. If one writes
the Hilbert series of $S_A$ in standard form $Q(t)/(1-t)^d$ then the
Hilbert series of $Q_A$ is just the polynomial $Q(t)$. In particular,
the highest degree of a non-vanishing element of $Q_A$ is the degree
of $Q(t)$.

In order to determine $\deg(Q(t))$ let
$E_P(m)=e_{d-1}m^{d-1}+\ldots+e_0$. Now in
\begin{gather*}
\sum_{m=0}^\infty E_P(m)t^m=\sum_{i=0}^{d-1}\left(e_i\cdot\sum_{m=0}^\infty
m^i\cdot t^m\right),
\end{gather*}
each term $\sum_{m=0}^\infty
m^i\cdot t^m$, for $m>0$, is a polylogarithm $\Li_{-i}(t)$ given by
$(t\frac{\de}{\de t})^n(\frac{t}{1-t})$. A simple calculation shows
that $\Li_{-i}(t)$ is the quotient of a polynomial of degree $i-1$ by
$(1-t)^i$. Hence the sum in the display is the quotient of a polynomial of
degree at most $d-1$ by  $(1-t)^d$. The degree is truly $d-1$ as one
can check from the differential expression for $\Li_{-i}(t)$ above. 

Therefore, the Hilbert series $Q(t)$ of $Q_A$ is a polynomial of
degree $d-1$. We have proved
\begin{thm}
Let $S_A$ be standard graded. 
The $b$-function for restriction of $M_A(\beta)$ to a generic point
$x+p=0$ divides $s(s-1)\cdots(s-k+1)$ where $k$ denotes the highest
degree in which the quotient $S_A/S_A\cdot(A\cdot\calE\cdot\del)$ is
nonzero. If, in addition,  $S_A$ is normal then one may take
$k=d$.\qedhere{\qed} 
\end{thm}

\def\scr{\mathcal}

\bibliographystyle{alpha}
\bibliography{bfu}

\end{document}